\documentclass[aip,reprint,longbibliography,groupedaddress,author-year,fleqn,a4paper,showkeys,notitlepage,12pt,onecolumn,tightenlines]{revtex4-1}

\usepackage[pass]{geometry} 
\usepackage{graphicx,mathrsfs}
\usepackage{epstopdf}
\usepackage{dsfont}
\usepackage[english]{babel}

\usepackage{amsmath,amsthm,amsfonts}

\theoremstyle{definition}

\theoremstyle{plain}

\newtheorem{prop}{Proposition}
\newtheorem{theorem}{Theorem}
\newtheorem{corollary}{Corollary}
\newtheorem{rem}{Remark}

\usepackage[T1]{fontenc}
\usepackage{lmodern}

\makeatletter
\def\p@subsection{}
\def\p@subsubsection{}
\makeatother

\usepackage{color} 
\definecolor{linkblue}{rgb}{0.,0.,0.5}
\usepackage[colorlinks=true, citecolor=linkblue, urlcolor=linkblue, linkcolor=black]{hyperref}

\begin{document}

\title{\large
A cumulant approach for the first-passage-time problem of the Feller square-root process
}

\author{Elvira Di Nardo}
\email[E-mail: ]{elvira.dinardo@unito.it}
\author{Giuseppe D'Onofrio}
\email[E-mail: ]{giuseppe.donofrio@unito.it}
\affiliation{Dipartimento di Matematica \lq G. Peano\rq, Università degli Studi di Torino, Via Carlo Alberto 10, Torino, 10123, Italy}

\begin{abstract}
The paper focuses on an approximation of the first passage time probability density function of a Feller stochastic process by using cumulants and a Laguerre-Gamma polynomial approximation. The feasibility of the method relies on closed form formulae for cumulants and moments recovered from the Laplace transform of the probability density function and using the algebra of formal power series. To improve the approximation, sufficient conditions on cumulants are stated. The resulting procedure is made easier to implement by the symbolic calculus and  a rational choice of the polynomial degree depending on skewness, kurtosis and hyperskewness. Some case-studies coming from neuronal and financial fields show the goodness of the approximation even for a low  number of terms. Open problems are addressed at the end of the paper.
\end{abstract}



\maketitle

\section{Introduction}

One-dimensional diffusion processes play a key role in the description of fluctuating phenomena belonging to different fields of applications as physics, biology,
neuroscience, finance and others \citep{karlin1981second,oks}. In particular, the class of stochastic processes with a linear drift and
driven by a Wiener process is widely used for its mathematical tractability and flexibility.
These models are described by a stochastic differential equation of the following type
\begin{equation}
\label{LIF_pearson}
{\rm d}Y_t=\left(-\tau Y_t+\mu \right) {\rm d}t+\Sigma(Y_t) \, {\rm d}W_t, \quad Y_0=y_0,  
\end{equation}
where $\{W_t\}_{t\geq 0}$ is a standard Wiener process,  $Y_0=y_0$ is the initial condition,  $\tau >0, \mu \in {\mathbb R}$ and the volatility
$\Sigma(Y_t)>0$ are such that a strong solution of Eq.\eqref{LIF_pearson} exists (\cite{Arnold} p.105).
 \smallskip \\ \indent
The volatility $\Sigma(Y_t)$ determines the amplitude of the noise and, according to its dependence on $Y_t$, it characterizes families of stochastic processes which are solution of Eq. \eqref{LIF_pearson}.
If 
\begin{equation}
\Sigma(Y_t)=\sqrt{aY_t^2+bY_t+c}, \quad a, b, c\in\mathbb{R},
\label{variance} 
\end{equation} 
the solution of Eq. \eqref{LIF_pearson} is called Pearson diffusion process \citep{forman2008}. The coefficients $a,b$ and $c$ are such that
the square root is defined for all the values of the state space $(B_1, B_2)$  of $Y_t,$  with $-\infty \leq B_1 < y_0< B_2 \leq +\infty$.
A wide range of well-known processes belongs to this class
($\sigma > 0$): 

\begin{description}
\item[{\rm \small Ornstein-Uhlenbeck process}] $a=b=0, c=\sigma^2$ and $\Sigma(Y_t)=\sigma;$ 
\item[{\rm \small  Inhomogeneous geometric Brownian motion}]
$a=\sigma^2$, $b=c=0$ and $\Sigma(Y_t)=\sigma  Y_t;$
\item[{\rm \small Jacobi diffusion}] $a=-\sigma^2, b= \sigma^2$ and $\Sigma(Y_t)=\sigma \sqrt{Y_t(1-Y_t)+c};$  
\item[{\rm \small Feller process (CIR model)}] $b=\sigma^2 , a=0$ and $\Sigma(Y_t)=\sigma \sqrt{Y_t+c}.$ 
\end{description}

Throughout this paper we will focus on this last process for its variety of applications not only in a biological context \citep{dit2006,fel51,sac95} but also in survival analysis, in the modeling of nitrous oxide emission from soil and in other applications such as physics and computer science (see \cite{dit2006} and references therein). In the mathematical finance it is known under the name of Cox-Ingersoll-Ross model (CIR) \citep{CIR}.

While general properties of the Feller process are well known since long, less known are properties related to first-passage-time (FPT) events which are very significant phenomena in all of the above mentioned situations. In this paper we consider the  dynamics of $Y_t$ until it crosses a threshold $S$ for the first time,  the so called (upcrossing) FPT, defined as 

\begin{equation}
\label{FPT}
T:=\inf \{t \geq 0: Y_t \geq S |0< y_0< S \}.
\end{equation}

Many contributions in the literature \citep{Giorno86,Yor2003,Linetsky,  Masoliver} focus on computing the Laplace transform (LT) of the probability density function (PDF) $g(t):=g(t|y_{0},S)$ of $T$, namely

\begin{equation}
\widetilde{g}(z)=\int_{0}^{\infty}e^{-z t} \, g(t) \, {\rm d}t, \quad z > 0.
\label{LTFPT1}
\end{equation}

The reason why the literature is focused on the LT of $g(t)$, is that the PDF is usually not known analitically and neither can be obtained by direct inversion of Eq. \eqref{LTFPT1}.
Nevertheless from $\widetilde{g}$ we can compute the probability of crossing the threshold $S$, $\mathbb P(T|y_0)=\int_{0}^{\infty}g(t){\rm d}t $,
and  the mean  FPT, $\mathbb E[T]$ as follows: 
\begin{equation}\label{rel_mom_lapl}
\mathbb P(T|y_0)=\widetilde{g}(z)\Big|_{z=0} \quad
 \hbox{\rm and} \quad \mathbb E[T] =- \frac{{\rm d} \widetilde{g}(z)}{{\rm d}z}\Bigg |_{z=0}.
\end{equation}

Moments of $T$ of any orders can be computed using higher derivatives of $\widetilde{g},$ when they exist.
As it is well-known, the moments of $T$ allow nice interpretation of statistical properties of the PDF $g(t)$ and of FPT events. A different strategy to get the moments of $T$ is using the transition PDF of the process $f(y,t|y_0,\tau)=\frac{\partial}{\partial y }\mathbb P(Y_t<y|Y_{\tau}=y_0).$ Indeed if $Y_t$ admits a stationary distribution $\mathcal{W}(y):=\lim\limits_{t\to\infty} f(y,t|y_0,0)$ independent of $y_0,$ the Siegert formula \citep{Siegert1951}  allows us to compute the  moments of $T$ as

\begin{eqnarray}\label{siegert}
\mathbb{E}[T^n]=n\int_{x_0}^{S}\frac{2 {\rm d} z}{[\Sigma(Y_t)]^2 \mathcal{W}(z)}\int_{-\infty}^{z}\mathcal{W}(x)\mathbb{E}[T^{n-1}] {\rm d}x, \quad n=1,2,\ldots.
\end{eqnarray} 
 \smallskip \\ \indent

Both the depicted strategies are impractical to compute the moments of $T$ for a Feller process. Despite the closed form formula of $\widetilde{g}(z)$ (see Section $2$), the computation of higher derivatives  is awkward and some efforts have  focused in evaluating just the mean and the variance of $T$ \citep{dit2006,DOnofrio} or at most the third moment \citep{Giorno88}. In terms of computational complexity, similar difficulties apply in computing moments of $T$ through Eq. \eqref{siegert}, although the stationary distribution of $Y_t$ is known to be a shifted gamma distribution (see Section 2). As the distribution of $T$ is often unavailable, simulations of the paths through Monte Carlo methods are still an efficient tool to get manageable estimations of $g(t),$ useful to analyze especially asymptotic properties. One more strategy consists in writing the FPT distribution as a Sturm-Liouville eigenfunction expansion series, first given for the Feller process in \cite{Linetsky}, using the classical argument of \cite{Kent1980} and \cite{Kent82}. Although this strategy provides an expression for the FPT density, information on the moments of $T$ can be obtained only numerically and refers exclusively to diffusion processes without natural boundaries.
A discussion on FPT of the Feller process in the presence of entrance, exit and reflecting boundary at the origin is given in \cite{Martin}, solving the Sturm-Liouville boundary problem in the case $\tau=0$.
 \smallskip \\ \indent

The goal of this paper is twofold: to give closed form formulae for the cumulants of $T$ of any order for the Feller process regardless of the nature of the boundaries and to give approximations of $g(t)$ by using moments recovered from cumulants. 
 \smallskip \\ \indent

Recall that if $T$ has moment generating function $\mathbb E[e^{zT}] < \infty$  for all $z$ in an open interval about $0,$ then its cumulants $\{c_k(T)\}_{k \geq 0}$ are such that 

\begin{equation}
\sum_{k \geq 1} c_k(T) 
\frac{z^k}{k!}  = \log \mathbb E[e^{zT}] 
\label{defcum}
\end{equation}
for all $z$ in some (possibly smaller) open interval about $0.$
Cumulants have nice properties compared with moments such as the semi-invariance and the additivity \citep{McCullagh}. Further properties on cumulants are given in Section 3.  Overdispersion and underdispersion as well as asymmetry and tailedness of the FPT PDF might be analized through the first four cumulants. Examples on how to employ the first four cumulants in the estimation of the parameters of a model fitted to data is given in \citep{Antunes, Seneta}.  
 \smallskip \\ \indent

The employment of cumulants in the FPT literature is not new \citep{Ramos}. However, their application has been limited to few cases and not in the direction addressed in this paper. Here, the idea to use cumulants essentially relies on the form of $\widetilde{g}(z)$ for the Feller process. Indeed $\widetilde{g}(z)$ is the ratio of two power series whose algebra is simplified if we consider $\log \widetilde{g}(z).$  We take advantage of the formal power series algebra \citep{EC} to give first
a closed form expression of $\{c_k(T)\}$ and then to recover moments. In Section 4, we propose to use a Laguerre series to approximate the PDF $g(t)$ taking into account the properties of $\widetilde{g}(z).$ To the best of our knowledge, this approach in evaluating the FPT PDF of the Feller process has not been investigate before in the literature.  Such an approximation works if moments (or cumulants) of $T$ are known and gives better results when the series is of Laguerre-Fourier type. As the PDF $g(t)$ is unknown, 
we give sufficient conditions on the cumulants of $T$ to guarantee the approximation with the Laguerre-Fourier series. We show how to take advantage of the formal power series algebra and of the symbolic calculus \citep{DiNardo} in implementing the proposed procedure. 
Some new results on the Kummer's function are also given.

 Then we apply our method to different case-studies inspired by 
 neuronal and financial models. One of the advantages of the method is that few terms are sufficient to have a good description of $g(t)$
 and the complexity of the overall computation is strongly reduced.
Statistical arguments motivate the choice of stopping the Laguerre series at the fifth term. The case-studies show that the resulting approximation is accurate also when the sufficient conditions are not completely fulfilled. A discussion section ends the paper, addressing future research and open problems.

\section{\label{sec:level2}The Feller process and the FPT problem}
We consider model \eqref{LIF_pearson} such that the function $\Sigma(\cdot)$ depends on the process itself and on $c \leq 0$. 
The Feller process investigated here is 
 given by 
\begin{equation}\label{LIF}
{\rm d} Y_t=\left( -\tau Y_t+\mu\right) {\rm d} t+\sigma\sqrt{Y_t-c} \,\, {\rm d}W(t).
\end{equation}
The state space of the process is the interval $(c,+\infty)$. 
The endpoints $c$ and $\infty$ can or cannot be reached in a finite time depending on the underlying parameters. According to the Feller classification of boundaries  \citep{karlin1981second}, $c$ is  an entrance boundary if it cannot be reached by $Y_t$ in finite  time, and there is no probability flow to the outside of the interval $(c,+\infty)$, that is, the process stays in $[c, +\infty)$ with probability 1. In particular, set $s:=2 (\mu-c \tau) /\sigma^2$. Then  $c$ is an entrance boundary  if $s\geq 1$.
\smallskip \indent \\ \indent
In the absence of a threshold, the Feller process admits a stationary distribution which is a shifted gamma distribution with the following shape, scale and location parameters
\begin{equation}\label{yinf_F}
Y_\infty \sim {\mbox{Gamma}}\left(s,
\frac{1}{2}\frac{\sigma^2}{\tau}, 
c\right).
\end{equation}
Let $Y_t$ evolve in the presence of a threshold $S$. 
Let $T$ be the FPT random variable of $Y_t$ through $S$ defined in Eq. \eqref{FPT}. Three distinct situations for the FPT can occur. Indeed, the process is said to be in the suprathreshold, subthreshold and threshold regimes if $\mathbb{E}[Y_\infty]>S, \mathbb{E}[Y_\infty]<S$ and $\mathbb{E}[Y_\infty]=S$, respectively, where the asymptotic mean of $Y_t$ is
\begin{eqnarray}\label{meanLIF_as}
\mathbb E[Y_{\infty}]=\lim_{t\rightarrow +\infty}\mathbb E[Y_t|y_0]=\frac{\mu}{\tau}.
\end{eqnarray}

The Siegert equation \citep{Masoliver}
\begin{eqnarray}\label{siegertdiff}
\frac{1}{2}\sigma^2(y_0-c)\frac{\partial^2 \widetilde{g}(z)}{\partial y_0^2}-\left(\tau y_0 + \mu \right)
\frac{\partial \widetilde{g}(z)}{\partial y_0}
-z \widetilde{g}(z)=0
\end{eqnarray}
with initial conditions $\widetilde{g}(z)=1$ if $y_0\equiv S$ and $\widetilde{g}(z)<+\infty$ for any $y_0$, provides the LT of the FPT PDF.
Indeed the solution of Eq. \eqref{siegertdiff} is
\begin{equation}
\label{Lapl_bo2}
\widetilde{g}(z)=\frac{\Phi\left(\frac{z}{\tau},
s,
\frac{2\tau(y_0-c)}{\sigma^2}\right)}
{\Phi\left(\frac{z}{\tau},
s,
\frac{2\tau(S-c)}{\sigma^2}\right)}, \quad z>0
\end{equation}
where $\Phi$ is the confluent hypergeometric function of the first kind (or Kummer's function) $\Phi(a,b,z)= {}_1F_1(a;b;z)$,
and 
\begin{equation}\label{hypergeo}
{}_pF_q(a_1, \ldots, a_p; b_1, \ldots, b_q;z):=\sum_{n\geq 0}\frac{ \langle a_1\rangle_n \cdots  \langle a_p\rangle_n}{ \langle b_1\rangle_n \cdots  \langle b_q \rangle_n}\frac{z^n}{n!}
\end{equation} 
is the generalized hypergeometric function, with $\langle a \rangle_n=a(a+1)\cdots (a+n-1), n\in \mathbb N$  the rising factorial and $\langle a \rangle_0=1$. For more details on Eqs. \eqref{yinf_F}-\eqref{Lapl_bo2} see \cite{DOnofrio}. 
In particular, the mean of $T$ is  \citep{Giorno88} 
\begin{eqnarray}
\label{meanFPT_F}
\mathbb E[T]=\frac{(S-y_0)}{\mu-\tau c}+\frac{1}{\tau} \sum_{n \geq 2} \frac{s^n \Gamma(s)}{n \Gamma(s+n)}
\frac{[(S-c)^n-(y_0-c)^n]}{(\frac{\mu}{\tau}-c)^n},      
\end{eqnarray} 
where ${\displaystyle \Gamma (z)=\int _{0}^{\infty }x^{z-1}e^{-x}\,{\rm d} x}$ is the gamma function.
 
\section{FPT cumulants}\label{s:methods}
Suppose $\widetilde{g}$ a formal power series  \citep{EC}
\begin{equation}
 \widetilde{g}(z)=\sum_{k \geq 0} \widetilde{g}_k \frac{z^k}{k!} \in {\mathbb R}[[z]]
 \label{fps}
\end{equation}
where ${\mathbb R}[[z]]$ denotes the ring of formal power series with coefficients in ${\mathbb R}.$ Then $\log \widetilde{g}(z)$ is  well defined 
\begin{equation}
\log \widetilde{g}(z)=\sum_{k \geq 1} c_k  \frac{z^k}{k!}
\label{cgf}
\end{equation}
and the coefficients $\{c_k\}_{k \geq 1}$ are named {\it formal cumulants} of $\{\widetilde{g}_k\}_{k \geq 0}.$ There are different formulae expressing formal cumulants in terms of $\{\widetilde{g}_k\}_{k \geq 0}$,  \cite{DiNardo}. Here we use the logarithmic (partition) polynomials $\{P_k\}$ such that
\begin{equation}
c_k = P_k(\widetilde{g}_1, \ldots, \widetilde{g}_{k}), \,\, k \geq 1,
\label{cummombis}
\end{equation}
 where
\begin{equation}
P_k(x_1, \ldots, x_{k}) = \sum_{j=1}^k (-1)^{j-1} (j-1)! B_{k,j}(x_1, \ldots, x_{k-j+1})
\label{logpol}
\end{equation}
and $\{B_{k,j}\}$ are the partial exponential Bell polynomials  \citep{EC}. Let us recall that, for a fixed positive integer $k$ and $j=1, \ldots,k,$ the $j$-th partial exponential Bell polynomial  in the variables  $x_1, x_2, \ldots, x_{k-j+1}$ is a homogeneous polynomial of degree $j$ given by
\begin{equation}
B_{k,j}(x_1,  \ldots, x_{k-j+1}) =  \sum \frac{k!}{\lambda_1! \lambda_2! \cdots  \lambda_{k-j+1}!} \prod_{i=1}^{k-j+1} \left(\frac{x_i}{i!}\right)^{\lambda_i}
\label{(parexpBell)}
\end{equation}
where the sum is taken over all sequences $\lambda_1, \lambda_2, \ldots, \lambda_{k-j+1}$ of non negative integers  such that 
\begin{equation}
\lambda_1 + 2 \lambda_2  + \cdots + (k-j+1) \lambda_{k-j+1}= k, \qquad \lambda_1 + \lambda_2 + \cdots + \lambda_{k-j+1}  = j.
\end{equation}
The $k$-th logarithmic polynomial \eqref{logpol} is a special case of the $k$-th general partition polynomial  
\begin{equation}
G_k( a_1, \ldots, a_k; x_1, \ldots, x_{k}) = \sum_{j=1}^k a_j B_{k,j}(x_1, \ldots, x_{k-j+1}), \,\, k \geq 1
\label{gpp}
\end{equation}
when $a_j=(-1)^{j-1} (j-1)!$ for $j \geq 1.$ The first five general partition polynomials $\{G_k\}_{k=1}^5$ are given in Table \ref{table1}. 
\begin{table}[ht]
\caption{General partition polynomials }  
\centering  
\begin{tabular}{c|l}  
\hline\hline  
$k\, $ & $\,\,G_k(a_1, \ldots, a_k; x_1, \ldots, x_{k})  $ \\ [0.5ex]  
\hline  
$1\, $ &  $\,\,a_1 x_1$ \\  
$2 \, $ &  $\,\,a_1 x_2 + a_2 x_1^2$ \\
$3 \, $ &  $\,\,a_1 x_3 + 3 a_2  x_2 x_1 + a_3 x_1^3$ \\
$4 \, $ &  $\,\,a_1 x_4 + 4 a_2 x_3 x_1 + 6 a_3 x_2 x^2_1 + a_4 x_1^4 + 3 a_2 x_2^2$ \\
$5 \, $ &  $\,\,a_1 x_5 + 5 a_2  x_4 x_1 + 10 a_ 2 x_3 x_2 + 10 a_3 x_3 x_1^2 + 15 a_3 x_2^2 x_1 + 10 a_4 x_2 x_1^3 + a_5 x_1^5$ \\ [1ex]  
\hline  
\end{tabular}
\label{table1}  
\end{table}
\smallskip \\
If we set $y_k=P_k(x_1, \ldots, x_{k})$ for $k \geq 1,$ then 
\begin{equation}
x_k=Y_k(y_1, \ldots, y_k) = \sum_{j=1}^k B_{k,j}(y_1, \ldots, y_{k-j+1}), \,\, k \geq 1
\label{(comexpBell)}
\end{equation}
are the inverse relations, with $\{B_{k,j}\}$ given in Eq. \eqref{(parexpBell)}. The polynomial $Y_k$ is the $k$-th complete  Bell (exponential) polynomial and is a special case of $G_k$  in Eq. \eqref{gpp} when $a_j=1$ for $j \geq 1.$ \\ \indent
The logarithmic and the complete Bell polynomials allow us to deal with moments and cumulants of $T$. Indeed if $\widetilde{g}$ is the Laplace transform of the PDF $g(t)$ and the rhs of  Eq. \eqref{fps} is its Taylor expansion about $0,$ then $\widetilde{g}_0=1,$ 
\begin{equation}
\widetilde{g}_k=(-1)^k \mathbb E[T^k], \,\, \hbox{$k \geq 1$}
\label{mom1}
\end{equation} 
and there exist cumulants of any order $\{c_k(T)\},$ see for instance \cite{Abate96}.  In particular, from Eq. \eqref{cgf} and Eq. \eqref{cummombis}, we have 
\begin{equation}
c_k =  (-1)^k  P_k(\mathbb E[T], \ldots, \mathbb E[T^k]) = (-1)^k c_k(T), \,\, \hbox{$k \geq 1$}.
\label{cummomter}
\end{equation}
Vice-versa, if cumulants $\{c_k(T)\}$ are known, moments of $T$ might be computed by using the inverse relations \eqref{(comexpBell)} 
\begin{equation}
\mathbb E[T^k] =  Y_k(c_1[T], \ldots, c_k[T]), \,\, 
\hbox{$k \geq 1$}
\label{momcum}
\end{equation}
or the well-known recursion formula  \citep{DiNardo2006}
\begin{equation}
\mathbb E[T^k] = c_k(T) + \sum_{i=1}^{k-1} \binom{k-1}{i-1}
c_{i}(T) \mathbb E[T^{k-i}].
\label{recumom}
\end{equation}

If $T$ is the FPT random variable of a Feller process modeled by Eq. \eqref{LIF}, the following theorem gives the closed-form expression of the $k$-th cumulant  for any order $k\geq 1.$ 
\begin{theorem} \label{Thm1}
The $k$-th FPT cumulant for the Feller process in Eq. \eqref{LIF} is
\begin{equation}
\!\!\!c_k(T)\!=\! \left(- \frac{1}{\tau} \right)^k  [c^*_k (y_0) -  c^*_k(S)]
\label{ck}
\end{equation}
where 
\begin{equation}
c^*_k (w) = P_k \left[ h_1\left(\frac{2\tau(w-c)}{\sigma^2}\right), h_2\left(\frac{2\tau(w-c)}{\sigma^2}\right), \ldots, h_k\left(\frac{2\tau(w-c)}{\sigma^2}\right)\right],
\label{cumF}
\end{equation}
$P_k$ is the $k$-th logarithmic polynomial \eqref{logpol} and
\begin{equation}
h_j(y) = j! \sum_{n \geq j} \left[{n\atop j}\right] \frac{y^n}{n! \langle s\rangle_n}, \,\,\,\, j=1,2,\ldots,k, \label{hj}
\end{equation}
with $\left[{n\atop j}\right]$ the unsigned Stirling numbers of first type
 and $\langle \cdot \rangle_n$ the $n$-th rising factorial.
\end{theorem}
\begin{rem}{\rm 
Note that \citep{EC}
\begin{equation}
\left[{n\atop j}\right]  = B_{n,j}\big(0!, 1!, \ldots,
(n-j+1)!\big), \,\, n \geq j
\label{unsigned}
\end{equation}
with $\{B_{n,j}\}$ the partial exponential Bell polynomials given in Eq. \eqref{(parexpBell)}.}
\end{rem}
\begin{proof}
In Eq. \eqref{Lapl_bo2}, set  $A = 2 \tau (y_0-c)/\sigma^2$ and $B = 2 \tau (S-c)/\sigma^2.$ From Eq. \eqref{cummomter} we get
\begin{equation}
\log \widetilde{g}(z) = \log \frac{{}_1F_1(\frac{z}{\tau};s;A)}{{}_1F_1(\frac{z}{\tau};s;B)} =\sum_{k \geq 1} (-1)^k c_k(T) \frac{z^k}{k!},
\label{proof1}
\end{equation}
where  
\begin{equation}
 {}_1F_1\left(\frac{z}{\tau};s;y\right)= \sum_{n \geq 0} \frac{\langle \frac{z}{\tau} \rangle_n}{\langle s \rangle_n}\frac{y^n}{n!}, \,\, y=A,B.
\label{log_iper}
\end{equation}
To expand the rhs of Eq. \eqref{log_iper} in formal power series in $z$, observe that 
\begin{equation}
\left\langle \frac{z}{\tau} \right\rangle_n=\sum_{j=0}^{n} \left[{n\atop j}\right] \frac{z^j}{\tau^j}
\label{stirling1}
\end{equation}
where $\left[{n\atop j}\right]$ are the unsigned Stirling numbers of the first type. Replacing Eq. \eqref{stirling1} in 
Eq. \eqref{log_iper}, after some algebra, we get
\begin{equation}
{}_1F_1\left(\frac{z}{\tau};s;y\right) = 1+\sum_{k\geq 1} 
\frac{z^k}{\tau^k} \left( \sum_{n \geq k} \left[{n\atop k}\right] \frac{y^n}{\langle s \rangle_n n!}\right).
\label{Hypergeom}
\end{equation}
From Eqs. \eqref{cgf} and \eqref{cummombis}, we get 
\begin{equation}
\log {}_1F_1\left(\frac{z}{\tau};s;y\right)  = \sum_{k \geq 1}
\frac{P_k[h_1(y), \ldots,h_k(y)]}{k!} \frac{z^k}{\tau^k} ,
\end{equation}
where $P_k$ is the $k$-th logarithmic polynomial given in Eq. \eqref{logpol} and $h_j(y)$ is given in Eq. \eqref{hj}.
Moreover Eq. \eqref{ck} follows 
taking into account Eq. \eqref{proof1} and by observing that
\begin{eqnarray*}
\!\!\!\!\!\!\!\!\!\!\! & & \log \frac{{}_1F_1\left(\frac{z}{\tau};s;A\right)}{{}_1F_1\left(\frac{z}{\tau};s;B\right)} =  \log {}_1F_1\left(\frac{z}{\tau};s;A\right) - \log {}_1F_1\left(\frac{z}{\tau};s;B\right) \nonumber \\
\!\!\!\!& = & \sum_{k \geq 1} \left(\frac{ P_k[ h_1(A),  \ldots, h_k(A)] -  P_k[ h_1(B), \ldots, h_k(B)]}{\tau^k} \frac{z^k}{k!} \right). \label{diff} 
\end{eqnarray*}
\end{proof} 
\begin{corollary}  The mean FPT  and the variance of $T$  are respectively
\begin{equation}
c_i(T) = \left( -\frac{1}{\tau} \right)^i \sum_{n \geq i} a_{i,n}  
\left( \frac{2 \tau}{\sigma^2} \right)^n \left[ (y_0-c)^n - (S-c)^n \right], \,\, i=1,2 \label{cum_12}
\end{equation}
where for $n \geq i$
\begin{equation}
a_{i,n} = \left\{ \begin{array}{lc}
\displaystyle{\frac{n^{-1}}{\, \langle s \rangle_n}}, & i=1 \\
\displaystyle{2 \frac{n^{-1} H_{n-1}}{\, \langle s \rangle_n} -
\sum_{k=1}^{n-1} \frac{k^{-1}}{\langle s \rangle_k}
\frac{(n-k)^{-1}}{\langle s \rangle_{n-k}}}, & i = 2 
\end{array} \right.
\end{equation}
with $H_{n-1}=\sum_{j=1}^{n-1} j^{-1}$ the harmonic numbers. 
\end{corollary}
\begin{proof}
The mean FPT  is obtained choosing $k=1$ in Eq. \eqref{ck} and observing that $\left[{n\atop 1}\right]=(n-1)!$ for $n \geq 1.$
The variance of the FPT  is obtained choosing $k=2$ in Eq. \eqref{ck}, 
 observing that $\left[{n\atop 2}\right]=(n-1)! H_{n-1}$ 
for $n \geq 2$ with $H_n$ the $n$-th harmonic number and
\begin{equation}
h_1(y)^2 = \sum_{n \geq 2} \left(\sum_{k=1}^{n-1}  a_{1,k} a_{1,n-k} \right) y^n.
\end{equation}
\end{proof}
\begin{rem}{\rm 
Observe that Eq. \eqref{cum_12} gives $\mathbb E[T]$ for $i=1$ and coincides with the expression \eqref{meanFPT_F} of the first moment of $T$. 
The comparison follows easily from the definition of $s$ and the property  $\langle s \rangle_n=\Gamma(s+n)/\Gamma(s).$}
\end{rem}
\begin{corollary}\label{cor_2}
 If $\{c_k(T)\}$ is the FPT cumulant sequence,
then 
\begin{equation}\label{mom}
\mathbb E[T^k] = \frac{(-1)^k}{\tau^k} \sum_{i=0}^k \binom{k}{i}  Y_{k-i}[c^*_1(y_0), \ldots, c^*_{k-i}(y_0)] Y_{i}[-c^*_1(S), \ldots,-c^*_i(S)]
\end{equation}
where $\{Y_i\}$ are the complete Bell polynomials given in Eq. \eqref{(comexpBell)} and $Y_0=1.$ 
\end{corollary}
\begin{proof}
Since $Y_k(a y_1, a^2 y_2, \ldots, a^k y_k) = a^k
Y_k(y_1, y_2, \ldots, y_k), a \in {\mathbb R},$ from 
Eq. \eqref{momcum} and Eq. \eqref{ck}, we get
\begin{equation}
\mathbb E[T^k] = \frac{(-1)^k}{\tau^k} Y_k[c^*_1 (y_0)-  c^*_1(S), \ldots, c^*_k (y_0)-  c^*_k(S)].
\end{equation}
Applying the binomial type property of the complete Bell polynomials, we have
\begin{eqnarray}
 & &\!\!\!\!\!\!\!\!\!\! Y_k[c^*_1 (y_0) -  c^*_1(S), c^*_2 (y_0) -  c^*_2(S), \ldots, c^*_k (y_0)- c^*_k(S)]  \nonumber \\
&= &  \sum_{i=0}^k \binom{k}{i} Y_{k-i}[c^*_1 (y_0), \ldots, c^*_{k-i}(y_0)] Y_{i}[-c^*_1(S), \ldots,-c^*_i(S)]
\label{binom}
\end{eqnarray}
and the result follows.
\end{proof}
Note that 
\begin{equation}\label{mom1}
Y_{i}[-y_1, \ldots, -y_i]  = \sum_{j=1}^i (-1)^j B_{i,j}[y_1, \ldots, y_{i+j-1}], \,\, i\geq 1
\end{equation}
since $B_{i,j}(a y_1, \ldots, a y_{i+j-1})=a^j B_{i,j}( y_1, \ldots,  y_{i+j-1}), a \in {\mathbb R}$ from Eq. \eqref{(parexpBell)}.
\subsection{{\small Computing FPT cumulants}}
For the subsequent applications of Theorem \ref{Thm1}, we add some remarks  on the efficiency of the implementation of Eq. \eqref{ck}. The logarithmic partition polynomials $\{P_k\},$
with $P_1(x_1)=x_1,$ might be generated by using the recurrence relation \citep{EC}
\begin{equation}
P_k(x_1, \ldots, x_k) = x_k - \sum_{r=1}^{k-1} \binom{k-1}{r} x_r
P_{k-r}(x_1, \ldots, x_{k-r}), \,\, k \geq 2.
\end{equation}
About the computation of $\{h_k(y)\}$ in Eq. \eqref{hj},
from Eq. \eqref{Hypergeom}, note that
\begin{equation}
h_k(y)=\left. \frac{{\rm \partial}^k}{{\rm \partial} u^k} \, {}_1\!F_1\left(u;s;y\right)\right\vert_{u=0}, \,\, k \geq 1.
\label{(aaa)}
\end{equation}
Derivatives of the Kummer's function with respect to the parameter $u$ have been computed in \cite{Ancarani}. The special case $u=0$ is given in terms of generalized Kamp\'e de F\'eriet-like hypergeometric functions.
An algorithm for the computation of the $k$-th derivative of the Kummer's function is given in \cite{Abad}. Here, we propose to use a standard implementation of the series in Eq. \eqref{hj}
involving the unsigned Stirling number of first type, as procedures implementing
the well-known triangular recurrence relation \citep{EC}
\begin{equation}
\left[{n+1}\atop{j}\right] = \left[ {n}\atop{j-1}\right] + n \left[{n}\atop{j}\right], \,\, j=1, \ldots, n+1, \, n \geq 0 \label{recuBnk}
\end{equation}
are available in many classical technical computing systems as {\tt Mathematica} or {\tt R}. 

From Eq. \eqref{(aaa)} it turns out that the usefulness of expression \eqref{hj} for the functions $h_j$ is twofold. It also constitutes an alternative way to express the derivative in Eq. \eqref{(aaa)} and so it can simplify the form of the Kamp\'e de F\'eriet function for particular values of the involved parameters. 
Moreover from Eq. \eqref{(aaa)} and the following expression \citep{abr}
\begin{equation}
\frac{d^k}{dx^k}f(x)=k! \sum_{n\geq k} \frac{(-1)^{n-k}\left[{n}\atop{k}\right]}{n!}\Delta^nf(x)
\end{equation}
we infer the following formula for the forward differences of order $n$-th of the Kummer function:
\begin{equation}
\Delta_u^n {}_1\!F_1\left(u;s;y\right) \big\vert_{u=0}=\frac{y^n}{\langle s \rangle_{n}}. 
\end{equation}

\section{The Laguerre-Gamma polynomial approximation}
The Edgeworth expansion is widely used in the literature to approximate a PDF around the Gaussian PDF, using a linear combination of Hermite polynomials with coefficients depending on the cumulants of the target PDF. To approximate a non-Gaussian PDF, a different family of polynomials is necessary together with a different reference density \citep{Asmussen}. If the target PDF $g(t)$ is unknown but expected to be close to some reference density $\varphi(t),$ then $\varphi(t)$ is used as a first approximation to $g(t)$ and later the approximation is improved by using suitable correction terms depending on a set of orthonormal polynomials. The following theorems show how to approximate the FPT PDF of a Feller process by using as reference density the gamma PDF with scale parameter $\alpha+1>0$ and shape parameter $\beta>0$ 
\begin{equation}
\varphi_{\alpha,\beta}(t) = \frac{\beta^{\alpha+1}}{\Gamma(\alpha+1)}
t^{\alpha} e^{- \beta t}, \,\, t > 0.
\label{(gammafun)}
\end{equation}
\begin{theorem} \label{ThApp}  Let
$a_k^{(\alpha)}={\mathbb E}[Q_k^{(\alpha)}(\beta T)], k \geq 0$ 
where
\begin{equation}
Q_k(t) = (-1)^k \left( \frac{\Gamma(\alpha+1+k)}{k! \, \Gamma(\alpha+1)}\right)^{-1/2} \!\!\! {L}_k^{(\alpha)}(t),
\label{(orthonormale)}
\end{equation}
and ${L}_k^{(\alpha)}(t)$ is the $k$-th generalized Laguerre polynomial 
\begin{equation}
{L}_k^{(\alpha)}(t) = \frac{\Gamma(\alpha+1+k)}{k!}
\sum_{j=0}^k \binom{k}{j} \frac{(-t)^j}{\Gamma(\alpha+j+1)}. 
\label{(genLag)}
\end{equation}
For $t > 0$ the series
\begin{equation}
U(\beta t,r):=\sum_{k \geq 0} a_k^{(\alpha)} Q_k^{(\alpha)}(\beta t)r^k
\label{(seriesU)}
\end{equation}
converges if $r \in (0,1)$ and 
\begin{equation}
\lim_{r \rightarrow 1} U(\beta t,r) = \frac{g(t)}{\varphi_{\alpha,\beta}(t)}. 
\label{(lim1)}
\end{equation}
\end{theorem}
\begin{proof}
Set $\beta t=w$ and observe that  in Eq. \eqref{(seriesU)} the series might be rewritten as
\begin{equation}
U(w,r)=\sum_{k \geq 0} b_k^{(\alpha)} {L}_k^{(\alpha)}(w)r^k, \,\, w > 0
\label{(seriesU1)}
\end{equation}
where for $k \geq 0$ 
\begin{equation}b_k^{(\alpha)} = \frac{\Gamma(k+1)}{\Gamma(k+1+\alpha)}
\int_0^{\infty} e^{-t} \, t^{\alpha} \, {L}_k^{(\alpha)}(t) \, f(t) \, {\rm d}t
\end{equation}
with $f(t)=g_\beta(t)/\varphi_{\alpha,1}(t), g_\beta(t)$ the PDF of $\beta T$ and $\varphi_{\alpha,1}(t)$ as given in
Eq. \eqref{(gammafun)}.
A sufficient condition to have the convergence of the series 
\eqref{(seriesU1)} for $r \in (0,1)$ at every point of continuity of $f(t)$ is \citep{Laguerre1}
\begin{equation}
\int_0^{\infty} e^{-zt} t^{\alpha} f(t) {\rm d}t = \Gamma(\alpha+1) \int_0^{\infty} e^{-(z-1)t} g_\beta(t) {\rm d}t  < \infty
\label{(convFeller)}
\end{equation}
for every $z - 1 >0,$ which is fulfilled when $T$ is the FPT random variable of a Feller process from Eq. \eqref{Lapl_bo2}. Therefore
\begin{equation}
\lim_{r \rightarrow 1} U(w,r) =  \frac{g_\beta(w)}{\varphi_{\alpha,1}(w)}
\label{(lim2)}
\end{equation}
and Eq. \eqref{(lim1)} follows from Eq. \eqref{(lim2)} after some algebra, replacing $w$ by $\beta t$  and recalling that 
$\beta g_\beta(\beta t)=g(t).$
\end{proof} 
Eqs. \eqref{(seriesU)} and \eqref{(lim1)} justify the
approximation of $g$ with the polynomial of degree $n$
\begin{equation}
\hat{g}(t) := \varphi_{\alpha,\beta}(t) \sum_{k=0}^{n} a^{(\alpha)}_k Q^{(\alpha)}_k(\beta t)
\label{approximationgen}
\end{equation}
for a suitable choice of $n,$ that we discuss in the next section. 
\begin{rem}
{\rm The polynomial approximation \eqref{approximationgen} is particularly suited when the PDF of $T$ is unknown, but its moments are available, as happens for the FPT random variable of the Feller process thanks to Corollary \ref{cor_2}. 
Indeed by observing that 
\begin{eqnarray*}
{\mathbb E}[Q_k^{(\alpha)}(\beta T)] & =  & (-1)^k \left( \frac{\Gamma(\alpha+1+k)}{k! \Gamma(\alpha+1)}\right)^{-1/2} \!\!\! \int_0^{\infty}  g(t) {L}_k^{(\alpha)}(\beta t) {\rm d} t \nonumber \\
 & = &  (-1)^k  \left( \frac{\Gamma(\alpha+1)\Gamma(\alpha+1+k)}{k!} \right)^{1/2}  \sum_{j=0}^k \binom{k}{j} \frac{(-\beta)^j {\mathbb E}(T^j)}{\Gamma(\alpha+j+1)}, 
\end{eqnarray*}
some algebra allows us to rewrite $\hat{g}(t)$ in Eq. \eqref{approximationgen} as
\begin{equation}
\hat{g}(t) = \beta (\beta t)^{\alpha} e^{- \beta t}  \sum_{k=0}^{n} A_k^{(\alpha)}{L}_k^{(\alpha)}(\beta t), \,\, t > 0
\label{approximationgen1}
\end{equation}
with coefficients 
\begin{equation}
A_k^{(\alpha)}  =  \sum_{j=0}^k \binom{k}{j} \frac{(-\beta)^j {\mathbb E}(T^j)}{\Gamma(\alpha+j+1)}, \,\, k=0,1, \ldots,n
\label{(coeffak1)}
\end{equation}
 depending on the moments of $T.$  Note that $\{A_k^{(\alpha)}\}$ might be expressed directly in terms of cumulants of $T$ by using Eq. \eqref{momcum}.}
\end{rem}

Sufficient conditions for the convergence of the series 
\begin{equation}
\sum_{k \geq 0} a_k^{(\alpha)} Q_k^{(\alpha)}(\beta t), \,\, t > 0
\label{(seriesU2)}
\end{equation}
can be recovered by using the analogous on the Laguerre series \citep{Laguerre1}. Indeed in such a case we have
$\lim_{r \rightarrow 1} U(\beta t,r) = U(\beta t,1)$ and  
\begin{equation}
g(t) = \varphi_{\alpha,\beta}(t) \sum_{k \geq 0}  a^{(\alpha)}_k Q^{(\alpha)}_k(\beta t), \,\, t >0.
\label{approximationgen2}
\end{equation}
The next corollary gives a sufficient condition on $g(t)$ to have
the series representation \eqref{approximationgen2}.
\begin{corollary} \label{sufcond}
The PDF $g(t)$ has the series representation \eqref{approximationgen2} if
\begin{equation}
\int_0^{\infty} t^{-\alpha} e^{\beta t} g(t)^2 {\rm d}t  <\infty. \label{(sufcond1)}
\end{equation}
\end{corollary}
\begin{proof}
Condition \eqref{(sufcond1)} is equivalent to ask $g(t)/\varphi_{\alpha,\beta}(t) \in {\mathscr L}^2(\nu)$, equipped with the usual inner product $<g_1,g_2>=\int g_1g_2d\nu$ and $\nu$ the measure having density $\varphi_{\alpha,\beta}(t).$ As $\nu$ admits moment generating function and all its moments are finite,  there exists a complete set of orthonormal polynomials in ${\mathscr L}^2(\nu),$ 
such that if $g/\varphi_{\alpha,\beta} \in {\mathscr L}^2(\nu),$ we may expand $g(t)/\varphi_{\alpha,\beta}(t)$
in terms of these polynomials. Let us observe that 
$\{Q_k^{(\alpha)}(\beta t)\}$  is a family of orthonormal polynomials in ${\mathscr L}^2(\nu)$  since $\{{L}_k^{(\alpha)}(t)\}$ is a family of orthogonal polynomials with respect to the weight function $t^{\alpha} e^{-t}.$   Therefore the Laguerre series \eqref{(seriesU2)} with $a_k^{(\alpha)}={\mathbb E}[Q_k^{(\alpha)}(\beta T)], k \geq 0$ represents the Fourier-Laguerre expansion of $g(t)/\varphi_{\alpha,\beta}(t)$ from whose uniqueness Eq. \eqref{approximationgen2} follows. 
\end{proof}
As $g(t)$ is unknown, it's not easy to verify directly the condition \eqref{(sufcond1)}. If 
\begin{equation}
\sum_{k=0}^{\infty} |{\mathbb E}[Q_k^{(\alpha)}(\beta T)]| < \infty
\label{Parseval1}
\end{equation}
then expansion \eqref{approximationgen2} holds, due to the Parseval identity. The accuracy of the approximation \eqref{approximationgen} depends upon the decay rate of $\{{\mathbb E}[Q_k^{(\alpha)}(\beta T)]\},$ as the ${\mathscr L}^2(\nu)$-loss is $\sum_{k=n+1}^{\infty} ( {\mathbb E}[Q_k^{(\alpha)}(\beta T)] )^2$ for a given order of truncation $n.$ A sufficient condition to have Eq. \eqref{(sufcond1)} is 
\begin{equation}
\left(|{\mathbb E}[Q_k^{(\alpha)}(\beta T)]|\right)^2 \approx k^{(-1 - \varepsilon)} \,\, \hbox{as $k \rightarrow \infty$ and $\varepsilon >0.$}
\end{equation}
So it is fundamental to have a good algorithm to evaluate the coefficients $\{{\mathbb E}[Q_k^{(\alpha)}(\beta T)]\}.$ 
This issue will be analyzed in the next paragraph.
\subsection{{\small Computational issues}}
To simplify the implementation, the approximating polynomial $\hat{g}(t)$ has been computed by using Eq. \eqref{approximationgen1}. The first five generalized Laguerre polynomials are given in Table \ref{TTable2}. 
\begin{table}[ht] 
\caption{Generalized Laguerre polynomials}  
\centering  
\begin{tabular}{c|l}  
\hline\hline  
$k\, $ & $\,\, {L}_k^{(\alpha)}(t) $ \\ [0.5ex]  
\hline  
$1\, $ &  $\,\, \langle \alpha +1 \rangle_1 - t$ \\  
$2 \, $ &  $\,\, \left(\langle \alpha +1 \rangle_2 - 2 \langle \alpha +2 \rangle_1 t + t^2 \right)/2!$ \\
$3 \, $ &  $\,\,\left(\langle \alpha +1 \rangle_3 - 3 \langle \alpha +2 \rangle_2 t + 3 \langle \alpha +3 \rangle_1 t^2 - t^3\right)/3!$ \\
$4 \, $ &  $\,\,\left(\langle \alpha +1 \rangle_4 - 4 \langle \alpha +2 \rangle_3 t + 6 \langle \alpha +3 \rangle_2 t^2 - 4 \langle \alpha + 4 \rangle_1 t^3 + t^4\right)/4!$ \\
$5 \, $ &  $\,\,\left(\langle \alpha +1 \rangle_5 - 5 \langle \alpha +2 \rangle_4 t + 10 \langle \alpha +3 \rangle_3 t^2 - 10 \langle \alpha + 4 \rangle_2 t^3 + 5 \langle \alpha + 5 \rangle_1 t^4 -  t^5\right)/5!$ \\ [1ex]  
\hline  
\end{tabular}\label{TTable2}
\end{table}
Many packages\footnote{See for example the package {\tt orthopolynom} in {\tt R}.} return the first $n$ generalized Laguerre polynomials
by using the following recursion formula \citep{EC} 
\begin{equation}
 L^{(\alpha)}_{k + 1}(t) = \frac{(2k + 1 + \alpha - t)L^{(\alpha)}_k(t) - (k + \alpha) L^{(\alpha)}_{k - 1}(t)}{k + 1}, \,\, k \geq 0,
 \label{recureq}
 \end{equation}
with ${L}_0^{(\alpha)}(t)  = 1.$ The same recursion \eqref{recureq} allows an efficient computation of the coefficients $\{A_k^{(\alpha)}\}. $ This result is proved in the following lemma where we use the symbolic calculus \citep{EC} formalized through the employment of a linear operator acting on a ring of polynomials, for details see \cite{DiNardo}.
\begin{prop}
Let $A_k^{(\alpha)}(y)=\frac{k!}{\Gamma(\alpha+1+k)} L^{(\alpha)}_{k}(y), $ for $k \geq 0.$ Then 
\begin{equation}
A_k^{(\alpha)} = \mathrm E[A_k^{(\alpha)}(\beta m)]
\end{equation}
where $\mathrm E$ is a linear operator transforming $m^j$ in $m_j=\mathbb{E}[T^j],$ that is $\mathrm E[m^j]=\mathbb{E}[T^j], j \geq 1$ and $\mathrm E[1]=m_0=1.$
\end{prop}
\begin{proof}
The result follows from Eq. \eqref{(coeffak1)}, since 
\begin{equation}
\mathrm E[A_k(\beta m)]= \sum_{j=0}^k \binom{k}{j} \frac{(-\beta)^j \mathrm E[m^j]}{\Gamma(\alpha+j+1)}
\end{equation}
and using the linear operator $\mathrm E.$ 
\end{proof}
Note that the moments $\{m_j\}$ are calculated from cumulants using  the recursive relation \eqref{recumom}.

The question of how to select the parameters $\alpha$ and $\beta$ in Eq. \eqref{approximationgen1} results to be a crucial point. A general guideline to their selection consists in matching the first two moments of $g(t)$  with the first two of 
$\varphi_{\alpha,\beta}(t).$ From a statistical point of view this choice mimics the well-known method of moments. From a computational point of view, if the first two moments of $g(t)$ and $\varphi_{\alpha,\beta}(t)$ coincide, then  $a_k^{(\alpha)}  
= A_k^{(\alpha)} = 0$ for $k = 1, 2$ simplifying the computation of  Eq. \eqref{approximationgen1} (\cite{Asmussen}). According to this rule, if 
\begin{equation}
\beta := \frac{c_1(T)}{c_2(T)} \qquad \hbox{\rm and} \qquad
\alpha := \beta {\mathbb E}[T]  -1 = \frac{c_1^2(T)}{c_2(T)}-1
\label{mommethodesti}
\end{equation}
 then $A_1^{(\alpha)}=A_1^{(\alpha)}=0$ in  Eq. \eqref{approximationgen1}. 

The choice of $\alpha$ and $\beta$ in Eq. \eqref{mommethodesti} deserves some deeper analysis. First note that the shape parameter $\alpha + 1$ is given by the mean of the scaled random variable  
$\beta T.$ Instead, the rate parameter $\beta$ measures  the inverse relative variance  of $g(t).$ The relative variance of a PDF is a normalized measure of its dispersion. Thus, the more spread out is $\varphi_{\alpha,\beta}(t)$ the greater is the underdispersion of $g(t).$ The next section gives some examples and applications of Eq. \eqref{approximationgen1} stopped at $n=5.$ The motivation of this choice stems from  the statistical meaning of the coefficients $\{A_k^{(\alpha)}\}.$  As $\alpha:={\mathbb E}[\beta T-1],$   symbolic calculus shows that the coefficients $\{A_k^{(\alpha)}\}$ are related to the $k$-th moment ${\mathbb E}[(\beta T-1)^k],$ without the normalizing constant $\Gamma(\alpha+j+1),$ which depends on the orthonormal property of $\{Q_k^{(\alpha)}(t)\}.$ Thus the third-order coefficient $A_3^{(\alpha)}$ accounts for the skewness of $g(t)$ while  the fourth-order coefficient $A_4^{(\alpha)}$ involves the weight of tails in causing dispersion, that is the kurtosis. The fifth-order coefficient $A_5^{(\alpha)}$ involves the hyper-skewness $m_5$ of $g(t)$  \citep{FinanceK}. Hyper-skewness measures the asymmetric sensitivity of the kurtosis, that is the relative importance of tails versus the center in causing skewness. Note that the sixth moment $m_6$ in $A_6^{(\alpha)}$ is the PDF hyper-kurtosis and measures both the peakedness and the tails compared with the normal distribution. As we are considering PDFs with  support $(0,\infty)$, the contribution of this coefficient is not statistically meaningful and not considered here. \\ \smallskip \\
Some suitable choices of the rate parameter $\beta$ might improve the approximation, as the following propositions show. 
\begin{prop} \label{Icond}
If $g/\varphi_{\alpha,\beta} \in {\mathscr L}^2(\nu),$ then $\beta < \frac{2}{{\mathbb E}[T]}.$ 
\end{prop}
\begin{proof}
As  $g/\varphi_{\alpha,\beta} \in {\mathscr L}^2(\nu),$ all the integrals in 
\begin{equation}
\int_0^{\infty} t^{-\alpha} e^{\beta t} g(t)^2 {\rm d}t  = 
\int_0^{1} t^{-\alpha} e^{\beta t} g(t)^2 {\rm d}t  + \int_1^{\infty} t^{-\alpha} e^{\beta t} g(t)^2 {\rm d}t = I_1 + I_2 \label{(sufcond2)}
\end{equation}
are finite. The FPT PDF $g(t)$ has exponential long-time behavior  with parameter the inverse  mean FPT \citep{Masoliver}, i.e
\begin{equation}
g(t)\approx\frac{1}{\mathbb{E}(T)}e^{-\frac{t}{\mathbb{E}(T)}}.
\end{equation}
Therefore to have the convergence of the latter integral in Eq. \eqref{(sufcond2)}, it is necessary to have $\beta -  \frac{2}{{\mathbb E}[T]} < 0.$
\end{proof} 
Thus in the following proposition we give a sufficient condition for the ratio $g/\varphi_{\alpha,\beta}$ to be in ${\mathscr L}^2(\nu),$
assuming 
\begin{equation}
\beta < \frac{2}{{\mathbb E}[T]} \quad \Longleftrightarrow \quad 2 c_2(T) > [c_1(T)]^2
\label{ccond}
\end{equation}
under the choice \eqref{mommethodesti}.
\begin{prop}
We have $g/\varphi_{\alpha,\beta} \in {\mathscr L}^2(\nu)$
if $\beta < 2/{\mathbb E}[T]$ and $g(t)=o(t^{\delta})$ with \\$2\delta+1 > [c_1(T)]^2/c_2(T).$ 
\end{prop}
\begin{proof}
Let us consider again Eq. \eqref{(sufcond2)}. Under condition 
\eqref{ccond}, the integral $I_2$ is always finite.
Moreover  $2 c_2(T) > [c_1(T)]^2 \iff -1<\alpha<1$.
Let us focus on $I_1.$ 
If $\alpha \in (0,1)$ and $g^2(t)=o(t^{2\delta})$ with $2 \delta > \alpha$ then the integrand function in $I_1$ is limited and $I_1$ is finite.
\end{proof}

\section{Examples}
As mentioned before, the Feller process plays a key role in a variety of applications. In this section we will investigate three examples coming from different areas of study.
\subsection{Example 1}
In the first example we consider dimensionless quantities to show the way the approximation is implemented and how it performs.

In Figure \ref{fig1}-top we show
the PDF of the first passage time through $S=1$ for the Feller process solution of Eq. \eqref{LIF} for $y_0=0.2$,  $c=0$, $\tau=1/1.5$, $\sigma=1$ and $\mu=0.9$. 
The curves are obtained using just $2,3,4$ or $5$ cumulants in the approximation method, i.e.
from Eq. \eqref{approximationgen1} for $n=2,3,4,5$. 
The approximations are compared to the  FPT PDF obtained through
simulation of $10^4$ first passage times of the process by discretization of Eq. \eqref{LIF} (see the \hyperref[appn]{Appendix}).
We observe that the agreement is satisfactory even for small $n$, altough the expression of the cumulants of any order is available and in principle can be used to improve the approximation.
The absolute error between the simulated PDF and the approximated one is shown in Figure \ref{fig1}-bottom. We observe that the error remains smaller than $0.05$ for $t>0$.
In $t=0$ the error is bigger, but the difference is due mainly to the error in the simulation rather than in the approximation. The reason is in the criteria for the bandwidth chosen in building the density from the histograms of the simulated first passage times.
\begin{figure}[htbp]
\begin{center}
\includegraphics[width=9cm]{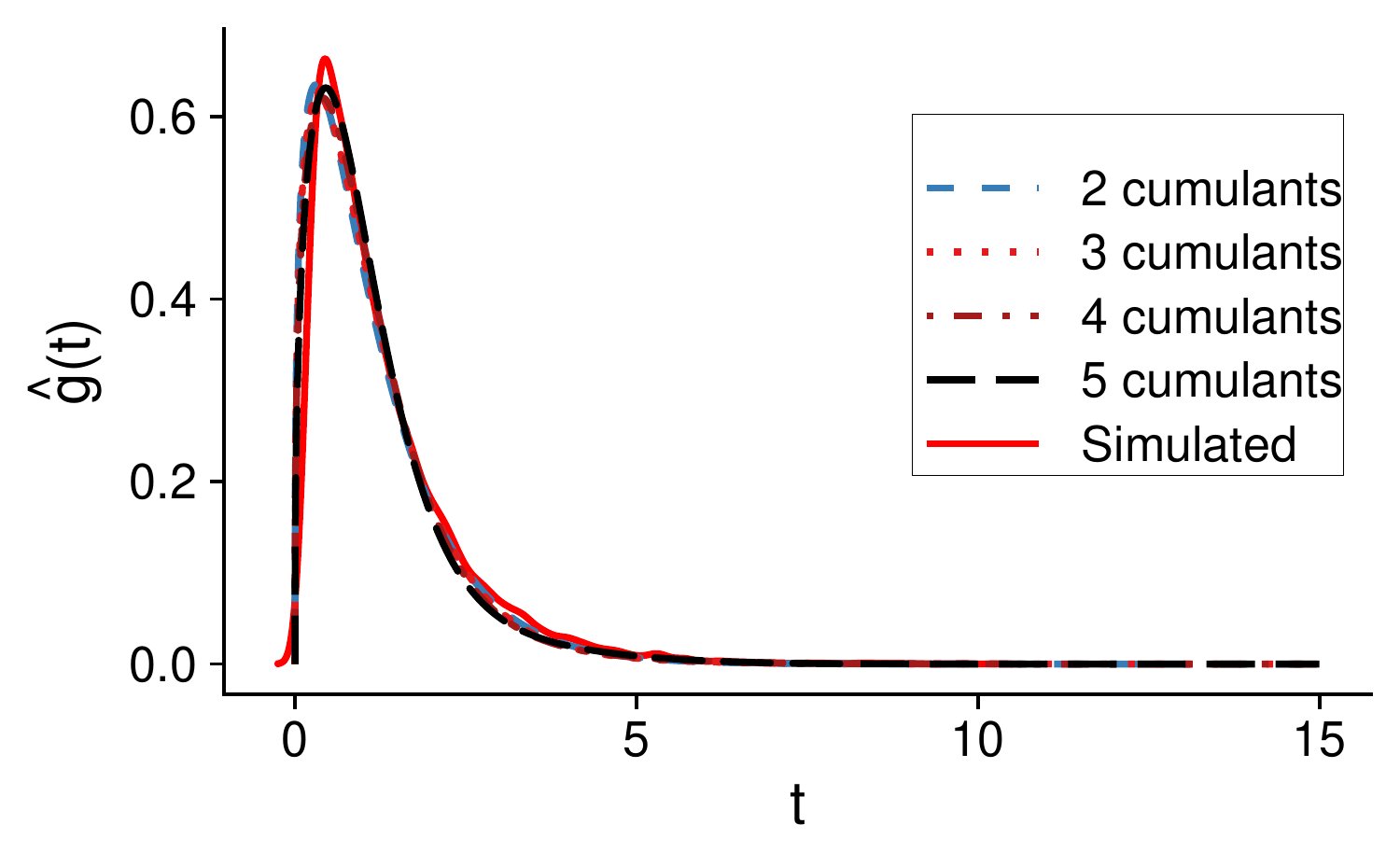}
\includegraphics[width=9cm]{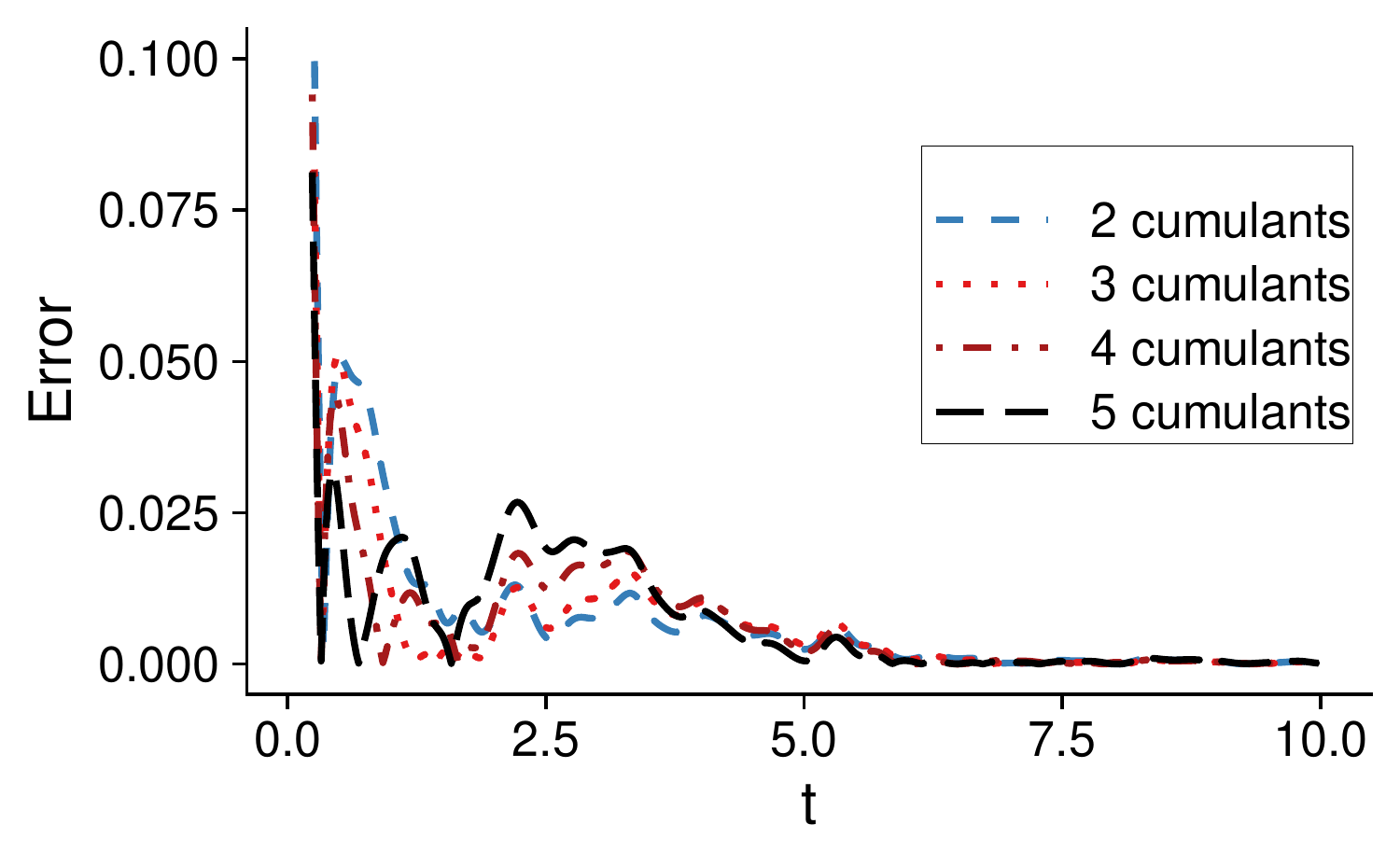}
\vspace*{-0.4cm}
\caption{TOP: Density of the first passage time through $S=1$ for the Feller process solution of Eq. \eqref{LIF} for $y_0=0.2$,  $c=0$, $\tau=1/1.5$, $\sigma=1$ and $\mu=0.9$. 
The curves are obtained from Eq. \eqref{approximationgen1} for $n=2,3,4,5$ (in the legend). The approximations are compared to the simulated FPT PDF (in solid-red).
The curve in red is built by simulation of $10^4$ first passage times of the Feller process obtained discretizing Eq. \eqref{LIF} by means of Eq. \eqref{mil_F1} with $\Delta t=10^{-2}$.
BOTTOM: The absolute error between the simulated PDF and the approximated one for the above cases. }
\end{center}\label{fig1}
\end{figure}

\subsection{Example 2: Neuronal modeling}
The Feller model was proved to fit experimental data of {\it in vitro} neurons under different conditions \citep{hopfner}.
For this reason in this subsection we  focus on an application to neuronal modeling of Eq. \eqref{LIF}.
The solution process $Y_t$ describes the evolution in time of the depolarization of the membrane potential of the neuron that is modelled 
as a leaky RC circuit with a drift characterizing the input stimuli.
Eq. \eqref{LIF}  describes the membrane depolarization 
until the occurence of a spike.
In accordance with the model, the spikes are generated when the process $Y_t$ crosses a voltage threshold $S$
for the first time, involving thus the FPT random variable.
The process is reset to the starting point $y_0$ after the spike and the evolution starts anew. 
In this framework
$y_0$ is the starting depolarization, $\sigma$ determines the amplitude of the noise, $c$ is the inhibitory reversal potential, $\tau$ is the inverse of the characteristic time constant of the neuron that takes into account the spontaneous voltage decay towards the resting potential in the absence of inputs and $\mu$  characterizes the input the neuron under consideration receives.
In the following, we consider the same parameters values used in  \cite{sac95},
the resetting potential is  equal to zero, i.e. $y_0=0$ mV,
the inhibitory reversal potential is fixed to $c=-10$ mV, the noise amplitude $\sigma=1.2$ $\sqrt{\mbox{mV}}/\sqrt{\mbox{ms}}$,  $\mu=3$  mV/ms and the firing threshold to $S=10$ mV.
The parameter of spontaneous decay is chosen $\tau=0.2$ ms (Figure \ref{fig_neur}).

\begin{figure}[htbp]
\begin{center}
\includegraphics[width=9cm]{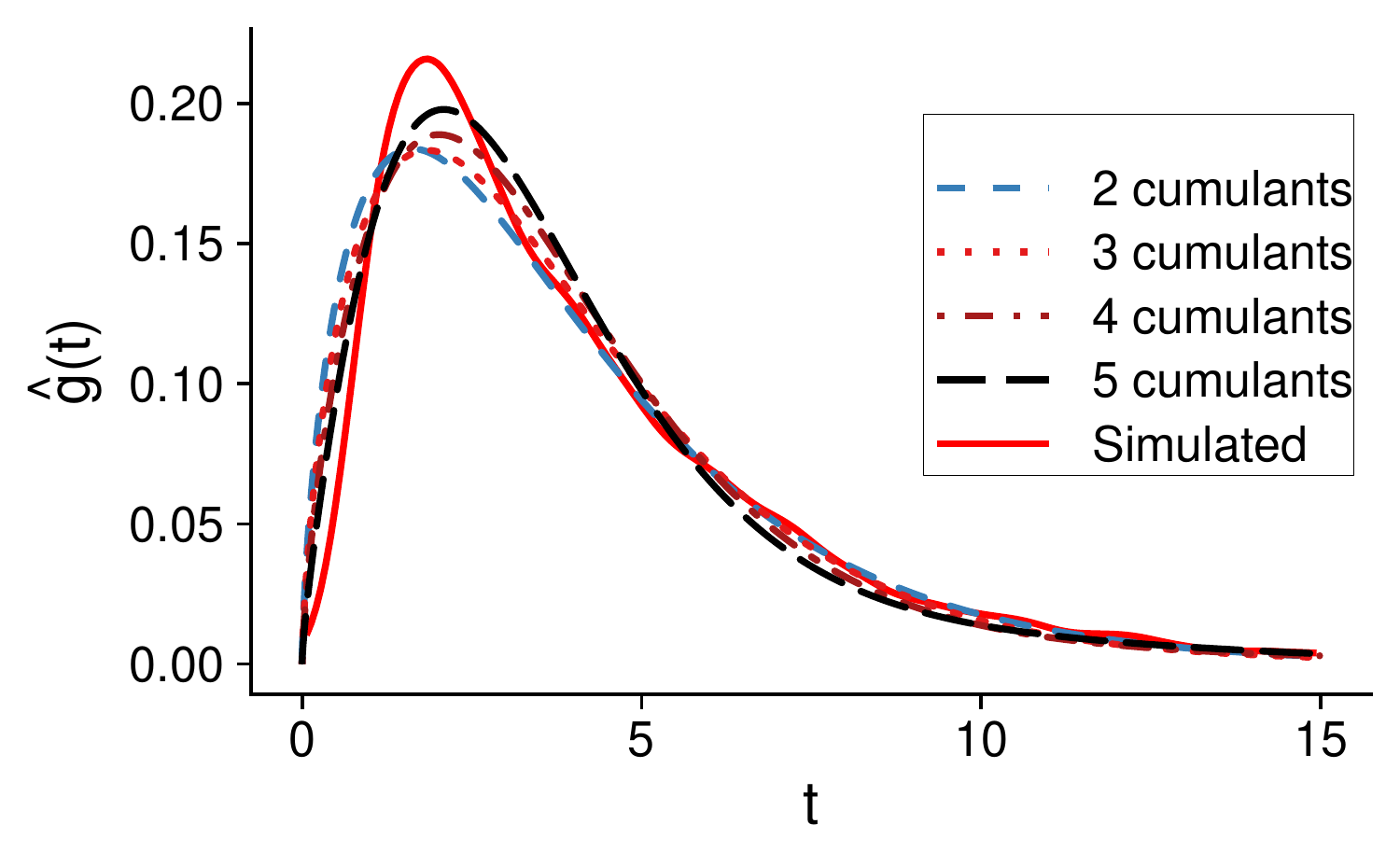}
\vspace*{-0.4cm}
\caption{Density of the first passage time through $S=10$ mV for the Feller neuronal model \eqref{LIF} for $y_0=0$ mV,  $c=-10$ mV, $\tau=0.2$ ms, $\sigma=1.2$ $\sqrt{\mbox{mV}}/\sqrt{\mbox{ms}}$ and $\mu=3$  mV/ms (the regime is suprathreshold). 
The curves are obtained from Eq. \eqref{approximationgen1} for $n=2,3,4,5$ (in the legend). The approximations are compared to the simulated FPT PDF (in solid-red).
}
\end{center}
\label{fig_neur}
\end{figure}

In \cite{dit2006} the noise amplitude is chosen $\sigma=2$ $\sqrt{\mbox{mV}}/\sqrt{\mbox{ms}}$ (Figure 5.3).
In this case the density is more skewed and the approximation fails to fit well the mode, altough the error remains of the order of $0.05$ (not shown).
\begin{figure}[htbp]
\begin{center}
\includegraphics[width=9cm]{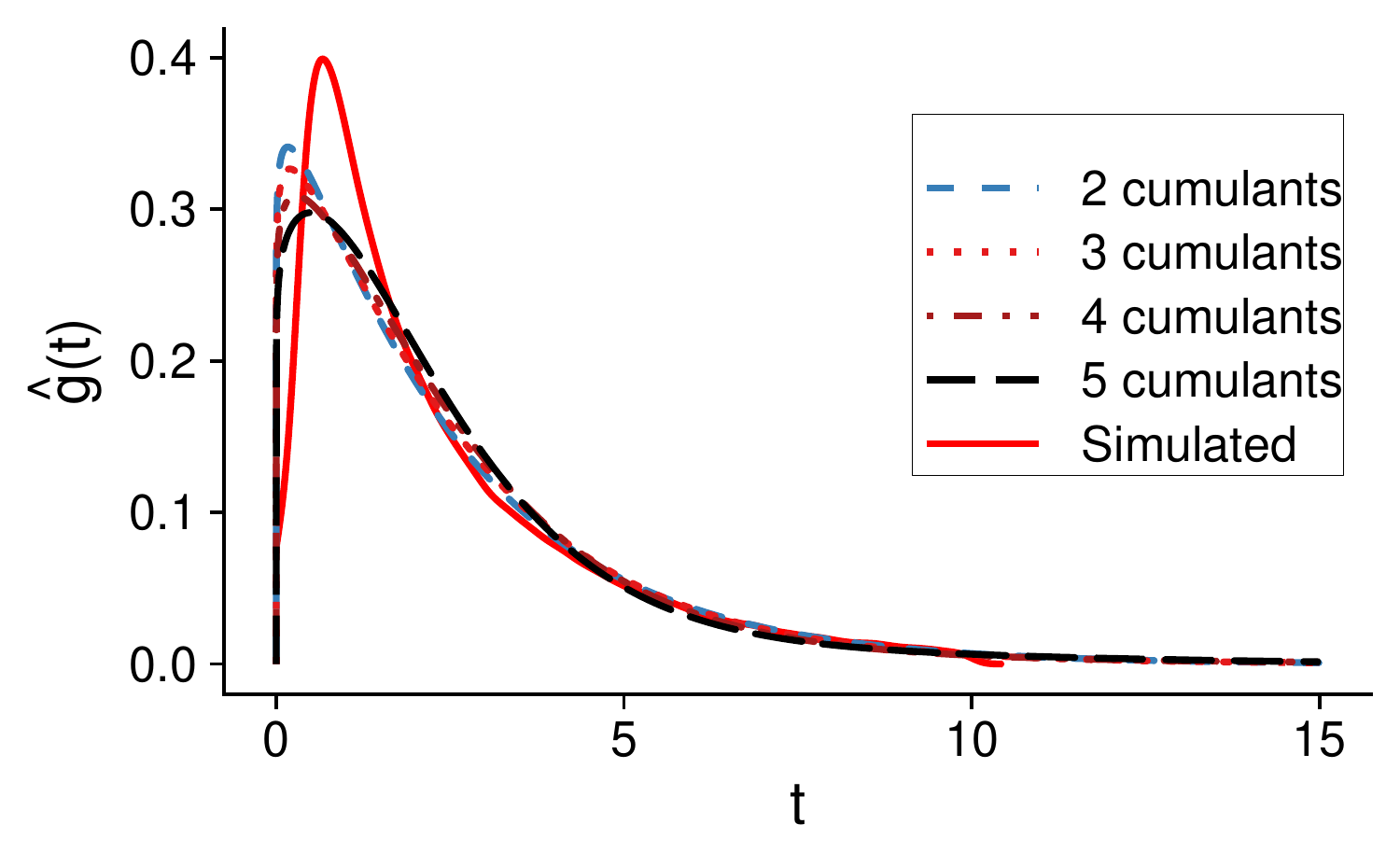}
\vspace*{-0.4cm}
\caption{Density of the first passage time  for the Feller neuronal model \eqref{LIF} for the same choice of parametrers of Figure \ref{fig_neur} except for $\sigma=2$ $\sqrt{\mbox{mV}}/\sqrt{\mbox{ms}}$ and $\mu=4$  mV/ms. The discretization step for the simulation is $\Delta t=10^{-3}$.
}
\end{center}
\label{fig_dit}
\end{figure}
The reason is in the properties of the gamma distribution that is our reference distribution: the mode is indeed not defined for $\alpha<1$,  ($\alpha=0.07$ in the example).

Another property of the gamma distribution to take care of is
the shape of the distribution for small $\alpha$. In fact in this case the gamma density stops to have the typical bell-shape, and $\varphi_{\alpha,\beta}(t)$ might fail a good approximation of $g(t)$ for small $t$.
This situation is presented in the following example.

\subsection{Example 3 : Financial mean-reverting models }
In  mathematical finance the Feller process goes under the name of CIR model  and it is used to
study the term structure of interest rates \citep{CIR} or mean-reverting models for a credit spread \citep{Linetsky}.
In the option pricing literature, the Feller process is used to describe the variance in models with stochastic volatility, where the most notable example is probably the Heston model \citep{heston,rouah}.
In this example we consider a stochastic model for
an instantaneous credit spread following Eq. \eqref{LIF} in $t\in[0.01,4]$ with the
long-run credit spread level of $200$ basis points $(\mu = 0.02\cdot 0.25)$, the initial spread level
of $100$bp ($y_0 = 0.01$), the rate of mean-reversion $\tau = 0.25$, and the volatility parameter $\sigma = 0.1$ (parameters given in \cite{Linetsky}). We
are interested in the first passage time density of the long-run level $S = 0.02$, starting from $y_0<S$.
Comparing the plot in Figure 5.4 with the one given by \cite{Linetsky} obtained with a different method of approximation, we observe a good asymptotic agreement.
We stress that the term asymptotic can be measliding since the agreement is good already for relatively small $t$.
 The PDF shape is not preserved for  $t \approx 0.01$. The reason is that for $\alpha<1$ ($\alpha=-0.34$ in this case)  the mode of the gamma distribution is not defined, and thus it cannot be matched with the one of $g(t)$, if it exists. 
However if we impose the teoretical information that $g(0)=0$, the behaviour of the PDF is reproduced.
Note that using Eq. \eqref{approximationgen1} we overcome the difficulties arisen from the simulation and the need to use $52$ terms in approximation expansion suggested by Linetsky.
\begin{figure}[htbp]
\begin{center}
\includegraphics[width=9cm]{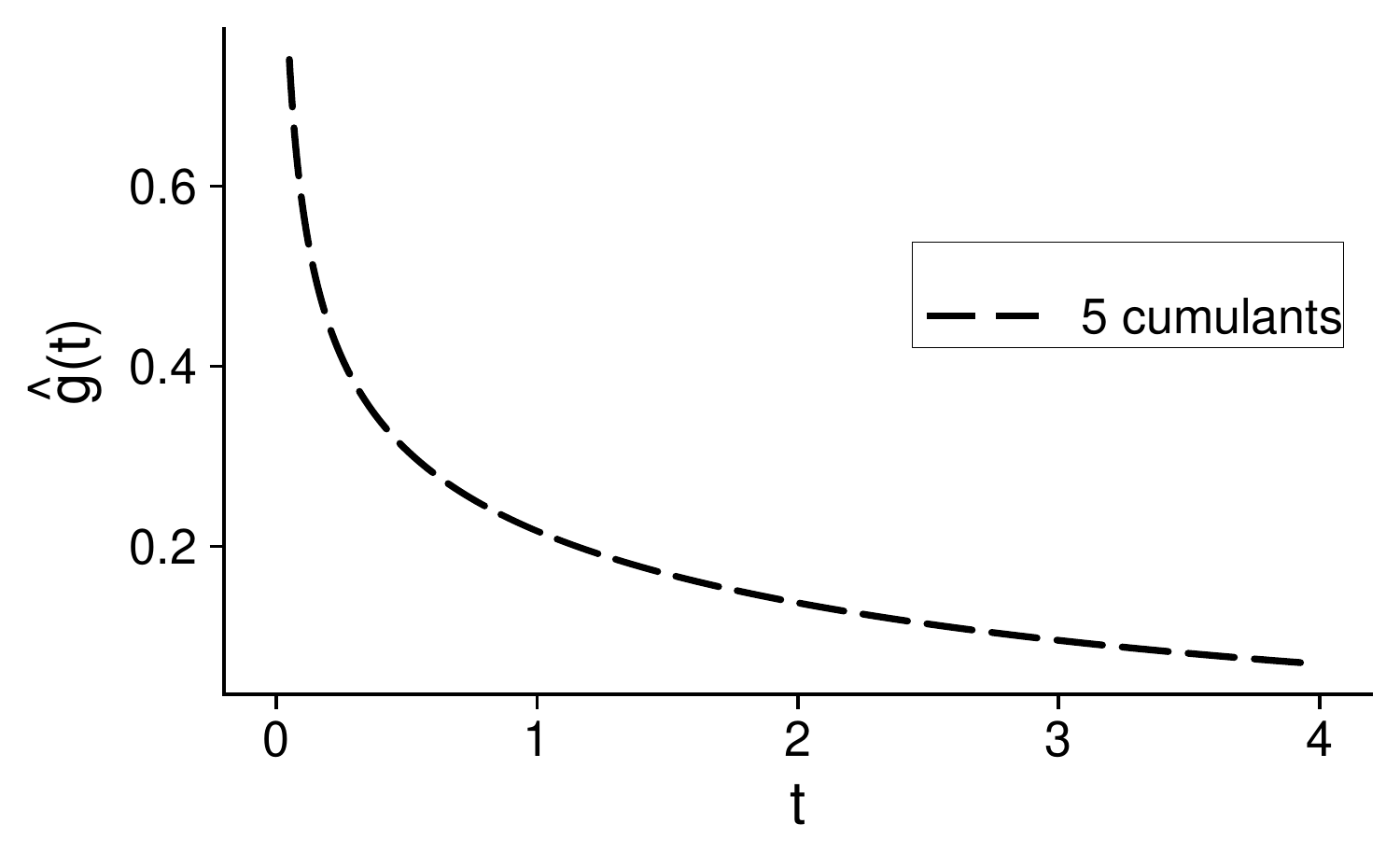}
\vspace*{-0.4cm}
\caption{Density of the FPT  for the CIR model \eqref{LIF} in the interval $t\in[0.01,4]$ as in \cite{Linetsky}, i.e. $\mu = 0.02\cdot 0.25$, $y_0 = 0.01$, $\tau = 0.25$, $\sigma = 0.1$ and  $S = 0.02$.
The curve is obtained from Eq. \eqref{approximationgen1} for $n=5$.
}
\end{center}
\label{fig_CIR}
\end{figure}

\section{Conclusions and open problems}
\label{s:disc}
We considered the well-known Feller stochastic process and the related FPT problem through a constant boundary. We provided a manageable closed form expression for the cumulants of $T$ of any order by which moments can be easily obtained, improving the current results available for the first three moments only. Note that the knowledge of higher moments gives qualitative information on the FPT PDF such as skewness, kurtosis, hyper-skewness and hyper-kurtosis.

We used cumulants to build a polynomial approximation of the FPT PDF $g(t),$ whose expression in closed form is still missing in the literature. The method is carried out involving the gamma distribution as a first approximation to $g(t)$ and then improving this approximation  by adding suitable correction terms based on a set of Laguerre polynomials. The resulting Laguerre-Gamma polynomial has coefficients whose computation was lightened by using the well-known recurrence relation of the Laguerre polynomials and the symbolic calculus. This computation was further simplified choosing the parameters of the gamma distribution with the method of moments. We have shown that the proposed method allows us to obtain good approximation of $g(t)$ even using a low degree ($5$ in the analyzed case-studies). Moreover it overcomes the difficulties arisen from the simulation for time $t$ close to zero. Some care must be taken when the PDF $g(t)$ is expected to have a mode differently from the  gamma distribution selected from the choice of its parameters. This circumstance deserves to be further investigated either in the choice of the parameters and in the expected properties of $g(t).$  Moreover, we give sufficient conditions to improve the approximation of the PDF $g(t)$ with the Laguerre-Gamma polynomial; criteria that are fulfilled in most cases of application. 

Future work includes the extension of this approach to other processes belonging to the class of Pearson's diffusion, since the expression of the Laplace transform of the FPT PDF for these processes is often written as a ratio of two hypergeometric functions. More in general, when the Laplace transform of the FPT PDF is a ratio of functions admitting a power series representation, cumulants might be recovered by using the algebra of formal power series and different polynomial approximations might be tested. For example if the transition PDF of the process has a power series representation of the Laplace transform $\widetilde{f}(z; x, y_0) = \int_{0}^{\infty}e^{-z t} f(x,t|y_0,0) {\rm d}t,$ thus the proposed method might be investigated as $\widetilde{g}(z)$
is again a ratio of power series, that is $\widetilde{g}(z)=\widetilde{f}(z; x, y_0)/\widetilde{f}(z; x, S).$

\appendix
\section{The Milstein method}\label{appn} 
To estimate the FPT PDF $g(t),$ we have implemented a classical Monte Carlo method and simulated the paths of $Y_t$ 
using  the stochastic differential equation \eqref{LIF}. The algorithm we refer relies on the Milstein scheme of 
discretization that is often used when the term $A_2$ of the SDE ${\rm d} Y_t = A_1(Y_t,t) {\rm d}t +\sqrt{A_2(Y_t,t)} {\rm d} W(t)$ depends on the process $Y_t$ (see for instance  \cite{KP}).
Truncation of the It\^{o}-Taylor expansion at the 
second order produces Milstein's method: 
\begin{eqnarray} \label{mil_F}
Y_{{n}}&=&Y_{n-1}+A_1(Y_{n-1})\Delta t+\sqrt{A_2(Y_{n-1})}\Delta W_{n-1} \nonumber \\
&&+{\frac  {1}{2}}\sqrt{A_2(Y_{n-1})}(\sqrt{A_2(Y_{n-1})})'\left[(\Delta W_{n-1})^{2}-\Delta t\right]
\end{eqnarray}
for $n=1,2,\ldots, N$ for some $N$. The Milstein scheme exhibits convergence of order $1$ in the strong sense and is a generalization of the Euler-Marayuma discretization scheme (the two methods coincide when $A_2(Y_t)$ does not depend on $Y_t$).
In case of Eq. \eqref{LIF}, Eq. \eqref{mil_F} gives
\begin{equation}\label{mil_F1}
Y_n=Y_{n-1}+\left(-\tau Y_{n-1}+\mu\right)\Delta t+\sigma \sqrt{Y_{n-1}-c} \ \Delta W_{n-1}+\frac{1}{4}\sigma^2\left[(\Delta W_{n-1})^{2}-\Delta t\right].
\end{equation}

\section*{Acknowledgements}
The second author was supported in part by Gruppo Nazionale per il Calcolo Scientifico (GNCS-INdAM).

\linespread{1}

%

\end{document}